\documentclass[reqno,11pt]{amsart}
\usepackage{eurosym}

\usepackage{amssymb}
\usepackage{tikz}
\usepackage{amsmath}
\usepackage{tikz}
\usepackage{pgflibraryplotmarks}
\usepackage{wasysym}

\theoremstyle{plain}
\newtheorem{theorem}{Theorem}[section]

\newtheorem{lemma}[theorem]{Lemma}

\theoremstyle{definition}

\newtheorem{remark}[theorem]{Remark}

\newtheorem*{Acknowledgement}{Acknowledgement}

\begin{document}
\def\sect#1{\section*{\leftline{\large\bf #1}}}
\def\th#1{\noindent{\bf #1}\bgroup\it}
\def\endth{\egroup\par}
\def\minf{\textnormal{minf}}
\def\spn{\textnormal{span}}
\def\N{\mathbb{N}}
\def\Z{\mathbb{Z}}
\def\R{\mathbb{R}}
\def\Q{\mathbb{Q}}
\def\w{\omega}
\def\K{\mathbb{K}}
\def\L{\mathcal{L}}
\def\D{\mathcal{D}}
\def\C{\mathcal{C}}

\title[The Hardy-Littlewood Maximal Operator on Discrete Morrey Spaces]
{The Hardy-Littlewood Maximal Operator on Discrete Morrey Spaces}
\author{Hendra Gunawan}
\address{Department of Mathematics, Bandung Institute of Technology, Bandung 40132,
Indonesia}
\email{hgunawan@math.itb.ac.id}
\author{Christopher Schwanke}
\address{Unit for BMI, North-West University, Private Bag X6001, Potchefstroom, 2520,
South Africa}
\email{schwankc326@gmail.com}
\date{\today}

\begin{abstract}
We discuss the Hardy-Littlewood maximal operator on discrete Morrey spaces of arbitrary
dimension. In particular, we obtain its boundedness on the discrete Morrey spaces using
a discrete version of the Fefferman-Stein inequality. As a corollary, we also obtain the
boundedness of some Riesz potentials on discrete Morrey spaces.
\end{abstract}

\maketitle

\section{Introduction}\label{S:intro}

While the Hardy-Littlewood maximal operator is well-known, discrete Morrey spaces were
only studied recently in \cite{GKS} (see also \cite{Ber} for related works). In this paper,
we investigate the boundedness of the (discrete) Hardy-Littlewood maximal
operator on discrete Morrey spaces of arbitrary dimension. Some important properties of this
operator (and many others) on the $\ell^p(\Z^d)$ spaces were discussed in \cite{Pierce}.
See also \cite{MSW, SW0,SW1,SW2,SW3} for related works on discrete analogues in harmonic
analysis.

The boundedness of the (continuous) Hardy-Littlewood maximal operator on the (continuous)
Morrey spaces was first studied in \cite{CF}, whose results were later extended in
\cite{Miz,Nak} to some generalizations of Morrey spaces. The driving force behind the
results in \cite{CF,Miz,Nak} is the so-called Fefferman-Stein inequality \cite[Lemma 1]{FS},
a result specifically regarding integrable functions defined on $\R^d$. We illustrate in
Theorem~\ref{T:FeffStein} how this inequality, despite its reliance on various tools only
available in the continuous setting, can be transformed into a natural discrete analogue.
As a consequence of Theorem~\ref{T:FeffStein}, we obtain the boundedness of the discrete
Hardy-Littlewood maximal operator on the discrete Morrey spaces in Theorem~\ref{T:Misbdd}.

We begin with some notation and definitions. First we set $\w := \N\cup\{0\}$ and use this
notation throughout the paper. For $m=(m_1,\dots,m_d)\in\Z^d$ and $N\in\w$ define
\[
S_{m,N} := \{k\in\Z^d: \|k-m\|_{\infty}\leq N\},
\]
where as usual $\|(x_1,\dots,x_d)\|_{\infty} := \max\{|x_i|:1\leq i\leq d\}$ for $(x_1,\dots,x_d)
\in\R^d$. Again following standard conventions, we denote the cardinality of a set $S$ by $|S|$.
Then we have $|S_{m,N}|=(2N+1)^d$ for all $m\in\Z^d$ and each $N\in\w$. Given $1\leq p\leq q
<\infty$ we define the discrete Morrey space $\ell^p_q(\Z^d)$ to
be the space of all functions $x\colon\Z^d\to\R$ for which
\[
\|x\|_{\ell^p_q(\Z^d)} := \underset{m\in\Z^d,N\in\w}{\sup}|S_{m,N}|^{\frac{1}{q}-\frac{1}{p}}
\biggl(\sum_{k\in S_{m,N}}|x(k)|^p\biggr)^{1/p}<\infty.
\]
By following the proof of \cite[Proposition 2.2]{GKS}, one can readily prove that
$\|\cdot\|_{\ell^p_q(\Z^d)}$ defines a norm on $\ell^p_q(\Z^d)$ and that $\ell^p_q(\Z^d)$
is a Banach space with respect to this norm. Indeed, \cite[Proposition 2.2]{GKS} proves
the given result for $d=1$, and its proof is easily adaptable to higher dimensions.

We wish to study the (discrete) Hardy-Littlewood maximal operator on these discrete Morrey
spaces of arbitrary dimension. To begin, define the discrete Hardy-Littlewood maximal
operator (or for emphasis, the ``odd'' discrete Hardy-Littlewood maximal operator) $M$ by
\[
Mx(m) := \underset{N\in\w}{\sup}\,\frac{1}{|S_{m,N}|}\sum_{k\in S_{m,N}}|x(k)|\quad
\big(x\in\R^{\Z^d},\ m\in\Z^d\big).
\]
The operator $M$ is a discrete analogue of the ``centered continuous'' Hardy-Littlewood
maximal operator, which is defined by
\[
\bar{M}f(y) := \underset{r>0}{\sup}\,\frac{1}{(2r)^d}\int_{Q_{y,r}}|f(z)|dz\quad
(f\in L_{loc}^1(\R^d),\ y\in\R^d),
\]
where $Q_{y,r} := \{t\in\R^d: \|t-y\|_\infty\leq r\}$. While the ``odd" discrete Hardy-Littlewood
maximal operator will be our primary interest, the following rendition of this function will
prove useful in obtaining a discrete analogue of the Fefferman-Stein inequality.

The ``even'' discrete Hardy-Littlewood maximal operator $\hat{M}$ is defined for
$x\in\R^{\Z^d}$ and $m=(m_1,\dots,m_d)\in\Z^d$ by
\[
\hat{M}x(m) := \underset{N\in\N}{\sup}\,\frac{1}{|R_{m,N}|}\sum_{k\in R_{m,N}}|x(k)|,
\]
where
\[
R_{m,N}:=S_{m,N}\setminus\{(k_1,\dots,k_d)\in\Z^d:k_i=m_i+N\ \text{for some}\ 1\leq i\leq d\}
\quad (m\in\Z^d,\ N\in\N),
\]
so that $|R_{m,N}|=(2N)^d$. Additionally, we define the ``uncentered'' discrete Hardy-Littlewood
maximal operator $\tilde{M}$ by
\[
\tilde{M}x(m) := \underset{S\ni m}{\sup}\,\frac{1}{|S|}\sum_{k\in S}|x(k)|\quad (x\in\R^{\Z^d},
\ m\in\Z^d),
\]
where the supremum above is taken over all sets of the form $S=S_{k,N}$, for some $k\in\Z^d$
and $N\in\w$, that contain $m$.

We say that two operators $T_1,T_2\colon \R^{\Z^d}\to\R^{\Z^d}$ are \textit{equivalent} if there exist
$C_1,C_2>0$ such that $C_1 T_1x(k)\leq T_2x(k) \leq C_2 T_1x(k)$ hold for all $x\in\R^{\Z^d}$ and
every $k\in\Z^d$. Regarding the operators $M, \hat{M}$, and $\tilde{M}$, we have the following lemma
which will be useful in our discussion in the next sections. We leave its proof to the reader.

\begin{lemma}\label{L:equiv}
The operators $M, \hat{M}$, and $\tilde{M}$ are pairwise equivalent.
\end{lemma}

\section{The Discrete Fefferman-Stein Inequality}\label{S:FS}

In this section, we provide a discrete version of the Fefferman-Stein inequality (\cite[Lemma~1]{FS}).
Theorem~\ref{T:FeffStein} below will be used to obtain the boundedness of the discrete Hardy-Littlewood
maximal operator on discrete Morrey spaces in Section 3.

We call a function $x\in\R^{\Z^d}$ \textit{positive} if $x(k)\geq 0$ for each $k\in\Z^d$.
Given $A\subseteq\R^d$, we denote the characteristic function of $A$ by $\chi_A$.

\begin{theorem}\label{T:FeffStein}
Let $1<p<\infty$. There exists $K>0$ such that for all $x\in\R^{\Z^d}$ and each positive $\phi\in\R^{\Z^d}$,
\begin{itemize}
\item[(1)] $\sum\limits_{k\in\Z^d}\bigl(Mx(k)\bigr)^p\phi(k)\leq
K\sum\limits_{k\in\Z^d}|x(k)|^{p}M\phi(k)$,
\item[(2)] $\sum\limits_{k\in\Z^d}\bigl(\hat{M}x(k)\bigr)^p\phi(k)\leq
K\sum\limits_{k\in\Z^d}|x(k)|^{p}\hat{M}\phi(k)$, and
\item[(3)] $\sum\limits_{k\in\Z^d}\bigl(\tilde{M}x(k)\bigr)^p\phi(k)\leq
K\sum\limits_{k\in\Z^d}|x(k)|^{p}\tilde{M}\phi(k)$.
\end{itemize}
\end{theorem}

\begin{proof}
We prove statement (2), from which statements (1) and (3) will follow from Lemma~\ref{L:equiv}.
For each $k=(k_1,...,k_d)\in\Z^d$ define the $d$-dimensional cube of volume one by
\[
\C_k := \{(y_1,...,y_d)\in\R^d:k_i\leq y_i<k_i+1\ \text{for all}\ 1\leq i\leq d\}.
\]
Next note that for every $a\in\R^{\Z^d}$ the function $\bar{a}$ defined by
\[
\bar{a}(t) := \sum_{k\in\Z^d}a(k)\chi_{\C_k}(t)\quad (t\in\R^d)
\]
is a member of $L_{loc}^1(\R^d)$ and
\[
\sum_{k\in R_{m,N}}a(k)=\int_{Q_{m,N}}\bar{a}(t)dt\quad (m\in\Z^d,\ N\in\w),
\]
where again $Q_{m,N}=\{t\in\R^d:\|t-m\|_\infty\leq N\}$. Therefore, we have
\[
\sum_{k\in\Z^d}a(k)=\int_{\R^d}\bar{a}(t)dt.
\]
Next let $x,\phi\in\R^{\Z^d}$, and suppose $\phi$ is positive. Then
$\bar{\phi}(t)\geq 0$ for every $t\in\R^d$, and in light of the remarks above,
\[
\sum_{k\in\Z^d}\bigl(\hat{M}x(k)\bigr)^p\phi(k)=\int_{\R^d}\sum_{k\in\Z^d}
\biggl(\underset{N\in\N}{\sup}\frac{1}{(2N)^d}\sum_{i\in R_{k,N}}|x(i)|\biggr)^p
\phi(k)\chi_{\C_k}(t)dt.
\]
Furthermore,
\begin{align*}
\int_{\R^d}&\sum_{k\in\Z^d}\biggl(\underset{N\in\N}{\sup}\frac{1}{(2N)^d}
\sum_{i\in R_{k,N}}|x(i)|\biggr)^p\phi(k)\chi_{\C_k}(t)dt\\
&=\sum_{k\in\Z^d}\int_{\C_k}\biggl(\underset{N\in\N}{\sup}\frac{1}{(2N)^d}
\int_{Q_{k,N}}|\bar{x}(s)|ds\biggr)^p\phi(k)dt.
\end{align*}
Note that for each $k\in\Z^d$ and all $t\in \C_k$ we have
\[
\int_{Q_{k,N}}|\bar{x}(s)|ds\leq\int_{Q_{t,N+1}}|\bar{x}(s)|ds.
\]
Hence
\begin{align*}
\sum_{k\in\Z^d}&\int_{\C_k}\biggl(\underset{N\in\N}{\sup}\frac{1}{(2N)^d}
\int_{Q_{k,N}}|\bar{x}(s)|ds\biggr)^p\phi(k)dt\\
&\leq 2^d\sum_{k\in\Z^d}\int_{\C_k}\biggl(\underset{N\in\N}{\sup}\frac{1}{(2(N+1))^d}
\int_{Q_{t,N+1}}|\bar{x}(s)|ds\biggr)^p\phi(k)dt\\
&\leq 2^d\int_{\R^d}\bigl(\bar{M}\bar{x}(t)\bigr)^p\bar{\phi}(t)dt.
\end{align*}
By the Fefferman-Stein inequality \cite[Lemma 1]{FS}, there exists $K>0$ such that
\[
\int_{\R^d}\bigl(\bar{M}\bar{x}(t)\bigr)^p\bar{\phi}(t)dt\leq
K\int_{\R^d}\left|\bar{x}(t)\right|^p\bar{M}\bar{\phi}(t)dt.
\]
Thus we obtain
\begin{align*}
\sum_{k\in\Z^d}\bigl(\hat{M}x(k)\bigr)^p\phi(k)&\leq 2^dK\int_{\R^d}\left|\bar{x}(t)\right|^p
\bar{M}\bar{\phi}(t)dt\\
&=2^dK\sum_{k\in\Z^d}\int_{\C_k}|x(k)|^p\left(\underset{r>0}{\sup}\frac{1}{(2r)^d}\int_{Q_{t,r}}\bar{\phi}(s)ds\right)dt.
\end{align*}
For each $k\in\Z^d$, let $k^\ast$ denote the midpoint of $\C_k$. Notice that for each $k\in\Z^d$ we have
\[
\int_{\C_k}|x(k)|^p\left(\underset{r>0}{\sup}\frac{1}{(2r)^d}\int_{Q_{t,r}}\bar{\phi}(s)ds\right)dt
\leq\int_{\C_k}|x(k)|^p\left(\underset{r>0}{\sup}\frac{1}{(2r)^d}\int_{Q_{k^\ast,r}}\bar{\phi}(s)ds\right)dt.
\]
Next suppose that $0<r<\frac{1}{2}$. Then for all $k\in\Z^d$,
\[
\frac{1}{(2r)^d}\int_{Q_{k^\ast,r}}\bar{\phi}(s)ds=\frac{1}{(2r)^d}\phi(k)r^d=
\frac{1}{2^d}\phi(k)\leq\frac{1}{(2\cdot\frac{1}{2})^d}\int_{Q_{k^\ast,\frac{1}{2}}}\bar{\phi}(s)ds.
\]
Hence
\begin{align*}
\sum_{k\in\Z^d}&\int_{\C_k}|x(k)|^p\left(\underset{r>0}{\sup}\frac{1}{(2r)^d}
\int_{Q_{t,r}}\bar{\phi}(s)ds\right)dt\\
&\leq\sum_{k\in\Z^d}|x(k)|^p\left(\underset{r\geq\frac{1}{2}}{\sup}\frac{1}{(2r)^d}
\int_{Q_{k^\ast,r}}\bar{\phi}(s)ds\right)\\
&\leq\sum_{k\in\Z^d}|x(k)|^p\left(\underset{r\geq\frac{1}{2}}{\sup}\frac{1}{(2r)^d}
\int_{Q_{k,\lfloor r+2\rfloor}}\bar{\phi}(s)ds\right)\\
&\leq 5^d\sum_{k\in\Z^d}|x(k)|^p\left(\underset{r\geq\frac{1}{2}}{\sup}\frac{1}{(2\lfloor r+2\rfloor)^d}
\int_{Q_{k,\lfloor r+2\rfloor}}\bar{\phi}(s)ds\right)\\
&\leq 5^d\sum_{k\in\Z^d}|x(k)|^p\left(\underset{N\in\N}{\sup}\frac{1}{(2N)^d}\sum_{i\in R_{k,N}}\phi(i)\right)\\
&=5^d\sum_{k\in\Z^d}|x(k)|^p\hat{M}\phi(k).
\end{align*}
Therefore,
\[
\sum_{k\in\Z^d}\bigl(\hat{M}x(k)\bigr)^p\phi(k)\leq (10)^dK\sum_{k\in\Z^d}|x(k)|^{p}\hat{M}\phi(k),
\]
as desired.
\end{proof}

\section{The Boundedness of the Discrete Maximal Operator}\label{S:bddnessofops}

We use the methods of F. Chiarenza and M. Frasca in \cite{CF} in this section, and we additionally require
the following lemma.

\begin{lemma}\label{bddness-M}
Let $1\leq p\leq q<\infty$. For any $x\in\ell^p_q(\Z^d)$ we have $Mx\in\ell^\infty(\Z^d)$ and $\|Mx\|_{\ell^\infty(\Z^d)}\leq\|x\|_{\ell_q^p(\Z^d)}$.
\end{lemma}

\begin{proof}
Let $x\in\ell^p_q(\Z^d)$, and put $m^\ast \in\Z^d$. Then
\[
Mx(m^\ast)=\underset{N\in\w}{\sup}\frac{1}{(2N+1)^d}\sum_{k\in S_{m^\ast,N}}|x(k)|
\leq\underset{m\in\Z^d,N\in\w}{\sup}\frac{1}{(2N+1)^d}\sum_{k\in S_{m,N}}|x(k)|.
\]
In \cite[Lemma 2.3]{GKS}, it is shown for $d=1$ that for any $m\in\Z^d$ and $N\in\w$,
\[
\frac{1}{(2N+1)^d}\sum_{k\in S_{m,N}}|x(k)| \leq \biggl(\frac{1}{(2N+1)^d}
\sum_{k\in S_{m,N}}|x(k)|^p\biggr)^{\frac{1}{p}}.
\]
One may readily check that the proof of \cite[Lemma 2.3]{GKS} also holds for general $d\in\N$. Hence
\begin{align*}
\underset{m\in\Z^d,N\in\w}{\sup}\frac{1}{(2N+1)^d}\sum_{k\in S_{m,N}}|x(k)|
&\leq \underset{m\in\Z^d,N\in\w}{\sup}(2N+1)^{-\frac{d}{p}}\biggl(\sum_{k\in S_{m,N}}
|x(k)|^p\biggr)^{\frac{1}{p}}\\
&\leq \underset{m\in\Z^d,N\in\w}{\sup}(2N+1)^{\frac{d}{q}-\frac{d}{p}}\biggl(
\sum_{k\in S_{m,N}}|x(k)|^p\biggr)^{\frac{1}{p}}\\
&= \|x\|_{\ell_q^p},
\end{align*}
which proves the lemma.
\end{proof}

We next present the main result of this paper.

\begin{theorem}\label{T:Misbdd}
Let $1<p\leq q<\infty$. For all $x\in\ell_{q}^{p}(\Z^d)$ we have $Mx\in\ell_{q}^{p}(\Z^d)$, and there exists $C>0$ such that $\|Mx\|_{\ell_{q}^{p}(\Z^d)}
\leq C\|x\|_{\ell^{p}_{q}(\Z^d)}$ holds for all $x\in\ell^p_q(\Z^d)$.
\end{theorem}

\begin{proof}
Let $m\in\Z^d$, and put $N\in\N$ (the case $N=0$ will be handled later). By
Theorem~\ref{T:FeffStein}~(1) there exists $K>0$ such that
\[
\sum_{k\in\Z^d}(Mx(k))^p\chi_{S_{m,N}}(k)\leq K\sum_{k\in\Z^d}|x(k)|^{p}M\chi_{S_{m,N}}(k),
\]
and thus
\begin{align*}
\sum_{k\in S_{m,N}}(Mx(k))^p&\leq K\sum_{k\in\Z^d}|x(k)|^{p}M\chi_{S_{m,N}}(k)\\
&=K\sum_{k\in S_{m,2N}}|x(k)|^{p}M\chi_{S_{m,N}}(k)+K\sum_{j=1}^{\infty}
\sum_{k\in S_{m,2^{j+1}N}\setminus S_{m,2^{j}N}}|x(k)|^{p}M\chi_{S_{m,N}}(k).
\end{align*}
Next note that for every $k\in\Z^d$ we have
\[
M\chi_{S_{m,N}}(k)=\underset{t\in\w}{\sup}\frac{1}{(2t+1)^d}\sum_{i\in S_{k,t}}
\chi_{S_{m,N}}(i)=\underset{t\in\w}{\sup}\frac{1}{(2t+1)^d}|S_{k,t}\cap S_{m,N}|.
\]
Let $j\in\N$, and assume $k\in S_{m,2^{j+1}N}\setminus S_{m,2^jN}$.
Then $\|k-m\|_{\infty}-N>0$. Now observe that
\begin{itemize}
\item[(1)] $S_{k,t}\cap S_{m,N}\neq\varnothing$ if and only if $\|k-m\|_{\infty}
\leq t+N$, that is $t\geq\|k-m\|_{\infty}-N$, and
\item[(2)] $S_{k,t}\cap S_{m,N}=S_{m,N}$ when $\|k-m\|_{\infty}\leq t-N$, that is
when $t\geq\|k-m\|_{\infty}+N$.
\end{itemize}
It follows from (1) and (2) above that
\begin{align*}
\underset{t\in\w}{\sup}\frac{1}{(2t+1)^d}|S_{k,t}\cap S_{m,N}|&=\underset{\|k-m\|_\infty-N
\leq t\leq\|k-m\|_\infty+N}{\sup}\frac{1}{(2t+1)^d}|S_{k,t}\cap S_{m,N}|\\
&\leq\frac{(2N+1)^d}{\bigl(2(\|k-m\|_\infty-N)+1\bigr)^d}\\
&\leq\left(\frac{3}{2}\right)^d\frac{N^d}{(\|k-m\|_\infty-N)^d}.
\end{align*}
Also observe that for every $k\in\Z^d$ we have $M\chi_{S_{m,N}}(k)\leq M$\textbf{1}$(k)=1$,
where \textbf{1} denotes the constant function on $\Z^d$ taking value one. Hence
\begin{align*}
K&\sum_{k\in S_{m,2N}}|x(k)|^{p}M\chi_{S_{m,N}}(k)+K\sum_{j=1}^{\infty}
\sum_{k\in S_{m,2^{j+1}N}\setminus S_{m,2^{j}N}}|x(k)|^{p}M\chi_{S_{m,N}}(k)\\
&\leq K\sum_{k\in S_{m,2N}}|x(k)|^{p}+\left(\frac{3}{2}\right)^dK\sum_{j=1}^{\infty}\sum_{k\in S_{m,2^{j+1}N}
\setminus S_{m,2^{j}N}}|x(k)|^{p}\frac{N^d}{(\|k-m\|_{\infty}-N)^d}.
\end{align*}
Now if $k\in S_{m,2^{j+1}N}\setminus S_{m,2^jN}$ then $\|k-m\|_{\infty}-N>2^jN-N\geq 2^{j-1}N$. Thus
\begin{align*}
K&\sum_{k\in S_{m,2N}}|x(k)|^{p}+\left(\frac{3}{2}\right)^dK\sum_{j=1}^{\infty}\sum_{k\in S_{m,2^{j+1}N}
\setminus S_{m,2^{j}N}}|x(k)|^{p}\frac{N^d}{(\|k-m\|_\infty-N)^d}\\
&\leq K\sum_{k\in S_{m,2N}}|x(k)|^{p}+\left(\frac{3}{2}\right)^dK\sum_{j=1}^{\infty}\sum_{k\in S_{m,2^{j+1}N}}
|x(k)|^{p}\frac{N^d}{(2^{j-1}N)^d}\\
&=K\sum_{k\in S_{m,2N}}|x(k)|^{p}+\left(\frac{3}{2}\right)^dK\sum_{j=1}^{\infty}\frac{1}{(2^d)^{j-1}}
\sum_{k\in S_{m,2^{j+1}N}}|x(k)|^{p}.
\end{align*}
Next observe that for every $t\in\Z^d$ and all $n\in\w$ we have
\[
\sum_{k\in S_{t,n}}|x(k)|^p\leq\|x\|_{\ell^p_q(\Z^d)}^p(2n+1)^{d-\frac{dp}{q}}.
\]
Hence
\begin{align*}
K&\sum_{k\in S_{m,2N}}|x(k)|^{p}+\left(\frac{3}{2}\right)^dK\sum_{j=1}^{\infty}\frac{1}{(2^d)^{j-1}}
\sum_{k\in S_{m,2^{j+1}N}}|x(k)|^{p}\\
&\leq K\|x\|_{\ell^p_q(\Z^d)}^p(4N+1)^{d-\frac{dp}{q}}+\left(\frac{3}{2}\right)^dK\sum_{j=1}^{\infty}
\frac{1}{(2^d)^{j-1}}\|x\|_{\ell^p_q(\Z^d)}^p(2^{j+2}N+1)^{d-\frac{dp}{q}}\\
&\leq 2^{d-\frac{dp}{q}}K\|x\|_{\ell^p_q(\Z^d)}^p(2N+1)^{d-\frac{dp}{q}}+\left(\frac{3}{2}\right)^dK
\sum_{j=1}^{\infty}\frac{(2^{j+1})^{d-\frac{dp}{q}}}{(2^d)^{j-1}}
\|x\|_{\ell^p_q(\Z^d)}^p(2N+1)^{d-\frac{dp}{q}}\\
&= 2^{d-\frac{dp}{q}}K\|x\|_{\ell^p_q(\Z^d)}^p(2N+1)^{d-\frac{dp}{q}}+
3^d(2^{d-\frac{dp}{q}})\left(\frac{1}{1-2^{-\frac{dp}{q}}}-1\right)
K\|x\|_{\ell^p_q(\Z^d)}^p(2N+1)^{d-\frac{dp}{q}}\\
&\leq C\|x\|_{\ell^p_q(\Z^d)}^p(2N+1)^{d-\frac{dp}{q}},
\end{align*}
where $\frac{C}{2}=\left(2^{d-\frac{dp}{q}}K\right)\vee\left(3^d(2^{d-
\frac{dp}{q}})\left(\frac{1}{1-2^{-\frac{dp}{q}}}-1\right)K\right)$.
Thus for every $m\in\Z^d$ and all $N\in\N$,
\[
(2N+1)^{\frac{d}{q}-\frac{d}{p}}\biggl(\sum_{k\in S_{m,N}}\bigl(Mx(k)\bigr)^p
\biggr)^{\frac{1}{p}}\leq C^{1/p}\|x\|_{\ell^p_q(\Z^d)}.
\]
That this inequality holds for $N=0$ (with $C=1$) follows from Lemma~\ref{bddness-M}.
Therefore,
\[
\|Mx\|_{\ell^p_q(\Z^d)}\leq(C^{1/p}\vee 1)\|x\|_{\ell^p_q(\Z^d)},
\]
which completes the proof.
\end{proof}

As an application of Theorem~\ref{T:Misbdd}, we obtain the boundedness of
some Riesz potentials on discrete Morrey spaces.

\begin{theorem}
Let $0<\alpha<d$ and $1<p<q<\frac{d}{\alpha}$. Define
\begin{equation}\label{FIO}
I_{\alpha}x(k)=\sum_{i\in\Z^d\setminus\{k\}}\frac{x(i)}{\|k-i\|_\infty^{d-\alpha}}
\quad (x\in\ell_q^p(\Z^d),\ k\in\Z^d).
\end{equation}
Set $s=\frac{dp}{d-\alpha q}$ and $t=\frac{qs}{p}$. Then $I_\alpha x\in \ell^s_t(\Z^d)$
for every $x\in\ell^p_q(\Z^d)$, and there exists a $C>0$ such that
\[
\|I_{\alpha}x\|_{\ell_t^s(\Z^d)}\leq
C\|x\|_{\ell_q^p(\Z^d)}\quad (x\in\ell_q^p(\Z^d)).
\]
\end{theorem}

\begin{proof}
Let $x\in\ell^p_q(\Z^d)$, and let $m\in\Z^d$. Then
\begin{align*}
Mx(m)&=\underset{N\in\w}{\sup}\frac{1}{(2N+1)^d}\sum_{k\in S_{m,N}}|x(k)|\\
&\leq |x(m)|\vee\underset{N\in\N}{\sup}\frac{1}{(2N)^d}\sum_{k\in S_{m,N}}|x(k)|\\
&\leq 2^d\Biggl(\frac{1}{2^d}\sum_{k\in S_{m,1}}|x(m)|\vee\underset{N\in\N}{\sup}
\frac{1}{(2N)^d}\sum_{k\in S_{m,N}}|x(k)|\Biggr)\\
&\leq 2^d\,\underset{r\geq 1}{\sup}\frac{1}{(2r)^d}\sum_{k\in\Z^d,\|m-k\|_\infty\leq r}|x(k)|.
\end{align*}
On the other hand,
\begin{align*}
\underset{r\geq 1}{\sup}\frac{1}{(2r)^d}\sum_{k\in\Z^d,\|m-k\|_{\infty}\leq r}|x(k)|&
\leq\underset{r\geq 1}{\sup}\frac{1}{(2\lfloor r\rfloor)^d}\sum_{k\in\Z^d,
\|m-k\|_\infty\leq\lfloor r\rfloor}|x(k)|\\
&=\underset{N\in\N}{\sup}\frac{1}{(2N)^d}\sum_{k\in S_{m,N}}|x(k)|\\
&\leq\left(\frac{3}{2}\right)^d Mx(m).
\end{align*}
Thus
\begin{equation}\label{eq3-3}
\left(\frac{2}{3}\right)^d\underset{r\geq 1}{\sup}\frac{1}{(2r)^d}\sum_{k\in\Z^d,\|m-k\|_\infty\leq r}
|x(k)|\leq Mx(m)\leq 2^d\underset{r\geq 1}{\sup}\frac{1}{(2r)^d}\sum_{k\in\Z^d,\|m-k\|_\infty\leq r}|x(k)|.
\end{equation}
Next let $r\geq 1$, and put $k\in\Z^d$. Then
\[
I_{\alpha}x(k)=\sum_{i\in\Z^d\setminus\{k\}}\frac{x(i)}{\|k-i\|_\infty^{d-\alpha}}
=\sum_{0<\|k-i\|_\infty\leq r}\frac{x(i)}{\|k-i\|_\infty^{d-\alpha}}+
\sum_{\|k-i\|_\infty>r}\frac{x(i)}{\|k-i\|_\infty^{d-\alpha}}.
\]
Define
\[
I_1:=\sum_{0<\|k-i\|_\infty\leq r}\frac{x(i)}{\|k-i\|_\infty^{d-\alpha}}\quad \text{and}\quad
I_2:=\sum_{\|k-i\|_\infty>r}\frac{x(i)}{\|k-i\|_\infty^{d-\alpha}}.
\]
Then
\begin{align*}
|I_1|\leq\sum_{j=0}^{\infty}\sum_{r2^{-j-1}<\|k-i\|_\infty\leq r2^{-j}}\frac{|x(i)|}{\|k-i\|_\infty^{d-\alpha}}.
\end{align*}
If $\|k-i\|_\infty>r2^{-j-1}$ then $\|k-i\|_\infty^{\alpha-d}<
r^{\alpha-d}2^{-j\alpha+jd-\alpha+d}$. Thus
\begin{align*}
\sum_{j=0}^{\infty}\sum_{r2^{-j-1}<\|k-i\|_\infty\leq r2^{-j}}\frac{|x(i)|}{\|k-i\|_\infty^{d-\alpha}}
&<\sum_{j=0}^{\infty}\sum_{r2^{-j-1}<\|k-i\|_\infty\leq r2^{-j}}|x(i)|r^{\alpha-d}2^{-j\alpha+jd-\alpha+d}\\
&\leq r^{\alpha}2^{d-\alpha}\sum_{j=0}^{\infty}2^{-j\alpha}\frac{1}{(r2^{-j})^d}
\sum_{0<\|k-i\|_\infty\leq r2^{-j}}|x(i)|.
\end{align*}
Let $J=\max\{j\in\w:r2^{-j}\geq 1\}$. Since $\sum\limits_{0<\|k-i\|_\infty\leq r2^{-j}}|x(i)|$
is an empty sum for all $j>J$, we have
\begin{align*}
r^{\alpha}2^{d-\alpha}\sum_{j=0}^{\infty}2^{-j\alpha}\frac{1}{(r2^{-j})^d}
\sum_{0<\|k-i\|_\infty\leq r2^{-j}}|x(i)|&=r^{\alpha}2^{d-\alpha}\sum_{j=0}^{J}2^{-j\alpha}\frac{1}{(r2^{-j})^d}
\sum_{0<\|k-i\|_\infty\leq r2^{-j}}|x(i)|.
\end{align*}
Using (\ref{eq3-3}), there exists a constant $C_0$ (for brevity, we do not record the precise
value of this constant) for which
\[
r^{\alpha}2^{d-\alpha}\sum_{j=0}^{J}2^{-j\alpha}\frac{1}{(r2^{-j})^d}
\sum_{0<\|k-i\|_\infty\leq r2^{-j}}|x(i)|\leq C_0r^{\alpha}Mx(k).
\]
Next note that
\[
|I_2|\leq\sum_{j=0}^{\infty}\sum_{2^jr<\|k-i\|_\infty\leq 2^{j+1}r}\frac{|x(i)|}{\|k-i\|_\infty^{d-\alpha}}.
\]
If $\|k-i\|_\infty>2^jr$ then $\|k-i\|_\infty^{\alpha-d}<(2^jr)^{\alpha-d}$. Hence we obtain
\begin{align*}
\sum_{j=0}^{\infty}&\sum_{2^jr<\|k-i\|_\infty\leq 2^{j+1}r}\frac{|x(i)|}{\|k-i\|_\infty^{d-\alpha}}
<\sum_{j=0}^{\infty}(2^jr)^{\alpha-d}\sum_{\|k-i\|_\infty\leq 2^{j+1}r}|x(i)|\\
&\leq\sum_{j=0}^{\infty}(2^jr)^{\alpha-d+\frac{d}{p}-\frac{d}{q}}\Biggl(
\sum_{\|k-i\|_\infty\leq 2^{j+1}r}1\Biggr)^{\frac{1}{p'}}(2^jr)^{\frac{d}{q}-\frac{d}{p}}
\Biggl(\sum_{\|k-i\|_\infty\leq 2^{j+1}r}|x(i)|^p\Biggr)^{\frac{1}{p}},
\end{align*}
where we use H\"older's inequality in the inequality above with reference to the H\"older
conjugate $p'$ of $p$. Moreover, we have
\begin{align*}
\sum_{j=0}^{\infty}&(2^jr)^{\alpha-d+\frac{d}{p}-\frac{d}{q}}\Biggl(
\sum_{\|k-i\|_\infty\leq 2^{j+1}r}1\Biggr)^{\frac{1}{p'}}(2^jr)^{\frac{d}{q}-\frac{d}{p}}
\Biggl(\sum_{\|k-i\|_\infty\leq 2^{j+1}r}|x(i)|^p\Biggr)^{\frac{1}{p}}\\
&\leq\sum_{j=0}^{\infty}(2^jr)^{\alpha-d+{\frac{d}{p}-\frac{d}{q}}}\Bigl(2^{j+2}r+
1\Bigr)^{\frac{d}{p'}}(2^j\lfloor r\rfloor)^{\frac{d}{q}-\frac{d}{p}}\Biggl(\sum_{\|k-i\|_\infty
\leq 2^{j+1}\lfloor r\rfloor}|x(i)|^p\Biggr)^{\frac{1}{p}}\\
&\leq 8^{\frac{d}{p}-\frac{d}{q}}\sum_{j=0}^{\infty}(2^jr)^{\alpha-d+{\frac{d}{p}-\frac{d}{q}}}
\Bigl(2^{j+2}r+1\Bigr)^{\frac{d}{p'}}(2^{j+2}\lfloor r\rfloor+1)^{\frac{d}{q}-\frac{d}{p}}
\Biggl(\sum_{i\in S_{k,2^{j+1}\lfloor r\rfloor}}|x(i)|^p\Biggr)^{\frac{1}{p}}\\
&\leq 8^{\frac{d}{p}-\frac{d}{q}}\sum_{j=0}^{\infty}(2^jr)^{\alpha-d+{\frac{d}{p}-\frac{d}{q}}}
\Bigl(2^{j+2}r+1\Bigr)^{\frac{d}{p'}}\|x\|_{\ell_q^p(\Z^d)}.
\end{align*}
It is readily checked that there exist constants $C_1, C_2>0$ such that
\begin{align*}
8^{\frac{d}{p}-\frac{d}{q}}\sum_{j=0}^{\infty}(2^jr)^{\alpha-d+{\frac{d}{p}-\frac{d}{q}}}
\Bigl(2^{j+2}r+1\Bigr)^{\frac{d}{p'}}\|x\|_{\ell_q^p(\Z^d)}&\leq C_1\sum_{j=0}^{\infty}
(2^jr)^{\alpha-d+{\frac{d}{p}-\frac{d}{q}}+\frac{d}{p'}}\|x\|_{\ell_q^p(\Z^d)}\\
&=C_2\|x\|_{\ell_q^p(\Z^d)}r^{\alpha-\frac{d}{q}}.
\end{align*}
Thus for $C_3=C_0\vee C_2$,
\begin{equation}\label{eq:3.3}
|I_{\alpha}x(k)|\leq C_3\Bigl(r^\alpha Mx(k)+r^{\alpha-\frac{d}{q}}\|x\|_{\ell_q^p(\Z^d)}\Bigr).
\end{equation}
Suppose for the moment that $k\in\Z^d$ satisfies $Mx(k)\neq 0$. By Lemma~\ref{bddness-M} we can
take $r:=\Bigl(\frac{\|x\|_{\ell_q^p(\Z^d)}}{Mx(k)}
\Bigr)^{q/d}\geq 1$ in (\ref{eq:3.3}) above and obtain
\begin{align*}
|I_{\alpha}x(k)|&\leq C_3\left[\Biggl(\biggl(\frac{\|x\|_{\ell_q^p(\Z^d)}}{Mx(k)}
\biggr)^{q/d}\Biggr)^\alpha Mx(k)+\Biggl(\biggl(\frac{\|x\|_{\ell_q^p(\Z^d)}}{Mx(k)}
\biggr)^{q/d}\Biggr)^{\alpha-\frac{d}{q}}\|x\|_{\ell_q^p(\Z^d)}\right]\\
&=2C_3\bigl(Mx(k)\bigr)^{1-\frac{\alpha q}{d}}\|x\|_{\ell_q^p(\Z^d)}^{\frac{\alpha q}{d}}.
\end{align*}
On the other hand, if $Mx(k)=0$ then $x=0$, and so $I_\alpha x(k)=0$. Thus the inequality
\[
|I_{\alpha}x(k)|\leq 2C_3\bigl(Mx(k)\bigr)^{1-\frac{\alpha q}{d}}\|x\|_{\ell_q^p(\Z^d)}^{\frac{\alpha q}{d}}
\]
holds in this case as well. Hence
\begin{align*}
\|I_\alpha x\|_{\ell_t^s(\Z^d)}&=\underset{m\in\Z,N\in\w}{\sup}\biggl(
\frac{1}{(2N+1)^{d-\frac{ds}{t}}}\sum_{k\in S_{m,N}}|I_\alpha x(k)|^s\biggr)^{\frac{1}{s}}\\
&\leq\underset{m\in\Z,N\in\w}{\sup}\biggl(\frac{1}{(2N+1)^{d-\frac{ds}{t}}}
\sum_{k\in S_{m,N}}\Bigl|2C_3\bigl(Mx(k)\bigr)^{1-\frac{\alpha q}{d}}\|x\|_{\ell_q^p(\Z^d)}^{\alpha q/d}
\Bigr|^s\biggr)^{\frac{1}{s}}\\
&=2C_3\|x\|_{\ell_q^p(\Z^d)}^{\alpha q/d}\underset{m\in\Z,N\in\w}{\sup}\biggl(
\frac{1}{(2N+1)^{d-\frac{dp}{q}}}\sum_{k\in S_{m,N}}\bigl(Mx(k)\bigr)^p\biggr)^{\frac{p}{ps}}\\
&=2C_3\|x\|_{\ell_q^p(\Z^d)}^{\alpha q/d}\|Mx\|_{\ell_q^p(\Z^d)}^{\frac{p}{s}}.
\end{align*}
By Theorem~\ref{T:Misbdd}, there exists $C>0$ such that
\[
2C_3\|x\|_{\ell_q^p(\Z^d)}^{\alpha q/d}\|Mx\|_{\ell_q^p(\Z^d)}^{\frac{p}{s}}
\leq C\|x\|_{\ell_q^p(\Z^d)}^{\alpha q/d}\|x\|_{\ell_q^p(\Z^d)}^{p/s}
=C\|x\|_{\ell_q^p(\Z^d)},
\]
as desired.
\end{proof}

\begin{remark}
The operator defined in (\ref{FIO}) may be considered as the discrete fractional
integral operator (for the continuous version, see for instance \cite{EGN}).
The proof that we presented above uses an analogue of Hedberg's inequality, which
we obtain right after we have inequality (\ref{eq:3.3}). The boundedness of this
operator on the $\ell^p(\Z^d)$ spaces can be found in \cite{SW2}.
\end{remark}

\begin{Acknowledgement}
This research was partially supported by the Claude Leon Foundation and by the DST-NRF Centre of Excellence in Mathematical and Statistical Sciences (CoE-MaSS) (second author). Opinions expressed and conclusions arrived at are those of the authors and are not necessarily to be attributed to the CoE-MaSS.
\end{Acknowledgement}

\end{document}